\documentclass[12pt,a4paper,psamsfonts]{amsart}
\usepackage{amssymb,amscd,amsxtra,calc, xypic}
\usepackage{cmmib57}

\theoremstyle{plain}
    \newtheorem{thm}{Theorem}[section]

    \newtheorem{corollary}[thm]{Corollary}
    
    \newtheorem{lemma}[thm]{Lemma}
    \newtheorem{proposition}[thm]{Proposition}
    \newtheorem{theorem}[thm]{Theorem}

\theoremstyle{definition}
    \newtheorem{definition}[thm]{Definition}

    \newtheorem{remark}[thm]{Remark}
\theoremstyle{remark}

\theoremstyle{question}

    \newtheorem{setup}[thm]{}

\newcommand{\C}{\mathbb{C}}

\newcommand{\PP}{\mathbb{P}}

\newcommand{\Z}{\mathbb{Z}}

\newcommand{\OO}{\mathcal{O}}

\newcommand{\Aut}{\operatorname{Aut}}

\newcommand{\Ker}{\operatorname{Ker}}

\newcommand{\Pic}{\operatorname{Pic}}

\newcommand{\Sing}{\operatorname{Sing}}

\usepackage{color}

\begin{document}

\title[Pseudo-automorphisms of positive entropy]{
Pseudo-automorphisms of positive entropy on the blowups of products of projective spaces
}

\author{Fabio Perroni}
\address
{
\textsc{Department of Mathematics, University of Bayreuth, Germany}\endgraf 
Current: 
\textsc{Scuola Internazionale Superiore di Studi Avanzati
(SISSA)}\endgraf
\textsc{Via Bonomea 265,
34136 Trieste, Italy}}
\email{fabio.perroni@sissa.it}

\author{De-Qi Zhang}
\address
{
\textsc{Department of Mathematics,
National University of Singapore}\endgraf
\textsc{
10 Lower Kent Ridge Road,
Singapore 119076,
Singapore}}
\email{matzdq@nus.edu.sg}

\begin{abstract}
We use a concise method to construct pseudo-automorphisms $f_n$ of the first dynamical degree
$d_1(f_n) > 1$
on the blowups of the projective $n$-space for all $n \ge 2$ and more generally on
the blowups of products of projective spaces.
These $f_n$, for $n=3$ have positive entropy, and for $n\geq 4$
seem to be the first examples of pseudo-automorphisms with $d_1(f_n) > 1$
(and of non-product type)
on rational varieties of higher dimensions.
\end{abstract}



\maketitle

\section{Introduction}\label{Intro}

We work over the field $\C$ of complex numbers.

A birational map $f: X \dashrightarrow X'$ of varieties
is a {\it pseudo-isomorphism} if it is an isomorphism outside
codimension-two closed subsets of $X$ and $X'$. If we assume further $X = X'$, then $f$ is called a
{\it pseudo-automorphism}.
By the minimal model program (which we will not use at all), a variety of dimension $\ge 3$ may have more than one
minimal models, but all of them are pseudo-isomorphic to each other.
In dimension two, every pseudo-automorphism of a normal projective surface is an automorphism, and all the minimal models
of a given surface are isomorphic to each other.

The main result of the paper is the following:

\begin{theorem}\label{ThA}
Let $w = w_{p, q, r}$ be the Coxeter element $($unique up to conjugation$)$
of the Weyl group $W(T_{p, q, r})$ {\rm (cf. \ref{Weyl})}.
Suppose that $r \ge 3$ and $\frac{1}{p}+\frac{1}{q}+\frac{1}{r}<1$.
Then
there exist a blowup
$$X = X_{p, q, r} \,\, \to \,\, (\PP^{r-1})^{p-1} = \PP^{r-1} \times \cdots \times \PP^{r-1}$$
at $q+r$ points $P_i$ lying on a  cuspidal curve $C \subset (\PP^{r-1})^{p-1}$ of multi-degree $r$
and a pseudo-automorphism
$f_w : X \,\, \dashrightarrow \,\, X$
such that $(f_w)^* | H^2(X, \Z)$ is equal to $w$.

In particular, the first dynamical degree $d_1(f_w)$
of $f_w$ is equal to the spectral radius $\rho(w)$ of $w$ and larger than $1$
$($cf. \cite{DS}$)$.
\end{theorem}

Here $C$ is the cuspidal curve of arithmetic genus one embedded in $(\PP^{r-1})^{p-1}$
by the product map $\Phi_{|D_1|} \times \cdots \times \Phi_{|D_{p-1}|}$
for some Cartier divisors $D_i$ of degree $r$ on $C$.
For instance, when $p = 2$, we can take
$$C = \{(1, z, z^2, ..., z^{r-2}, z^{r}) \, : \, z \in \C \}$$
in affine coordinates.

When $p = 2$ and $n = q+r$, we can take
$w = (12 \cdots n) r_{I,1}$, where the permutation is on the part $e_j$ of the
standard basis of the hyperbolic lattice
$\Lambda_n = h_1 \Z + \sum_{j=1}^{n} e_j \Z$
(naturally identified with $H^2(X, \Z)$)
and $r_{I,1}$ is the reflection corresponding to the root $\alpha_{I, 1} = h_1 - \sum_{j=1}^r e_j$ (cf.~\ref{Weyl}).

As a consequence of Theorem \ref{ThA} and Corollary \ref{h254} late on, we have:

\begin{corollary}\label{Cor1}
\begin{itemize}
 \item [(1)]
When $\{p, q, r\} = \{2, 3, 7\}$ $($as unordered sets$)$ and $r \ge 3$,
$f_w$ is a pseudo-automorphism of the blowup of $(\PP^{r-1})^{p-1}$
at $q+r$ points and $d_1(f_w) = 1.17628 \dots$ is the Lehmer number of
the Lehmer polynomial
$$x^{10} + x^9 - (x^7 + x^6 + x^5 + x^4 + x^3) + x + 1 .$$
 \item [(2)]
When $\{p, q, r\} = \{2, 4, 5\}$ $($as unordered sets$)$ and $r \ge 3$,
$f_w$ is a pseudo-automorphism of the blowup of $(\PP^{r-1})^{p-1}$
at $q+r$ points and $d_1(f_w) = 1.28064 \dots$ is
the largest root of the Salem polynomial
$$x^8 - x^5  - x^4 - x^3 + 1 .$$
 \item [(3)]
When $\{p, q, r\} = \{3, 3, 4\}$ $($as unordered sets$)$,
$f_w$ is a pseudo-automorphism of the blowup of $(\PP^{r-1})^{p-1}$
at $q+r$ points and $d_1(f_w) = 1.40127 \dots$ is
the largest root of the Salem polynomial
$$x^6 - x^4 - x^3 - x^2 + 1 .$$
\item[(4)]
If $(p, q, r) = (2, q, 4)$ and $q \ge 5$, the topological entropy
$h(f_w) = \log d_1(f_w) > 0$.
\end{itemize}
\end{corollary}

The three types of $T_{p, q, r}$ in (1), (2) and (3) above are the only $T$-shaped minimal hyperbolic Coxeter diagrams
(cf.~\cite[Table 5]{Mc02}).
The three Salem numbers above are the smallest Salem numbers of
degrees $10$, $8$ and $6$, respectively. Hence one also realizes
the Lehmer number as $d_1(f_w)$ of the pseudo-automorphism of $X$ (a $10$-point blowup of $\PP^6$).

We remark that $h(f_w) = \log 1.28064 \dots$ is the smallest known topological entropy $(> 0$)
of a pseudo-automorphism on a rational threefold which is not of product type. In \cite{BK11},
the authors have constructed a pseudo-automorphism $f$ on the blowup of
$\PP^3$ at $2$ points and $13$
curves with $h(f) = \log 1.28064 \dots$. Our construction is different from theirs;
for instance, $f$ is induced by a quadratic birational map on $\PP^3$, while the $f_w$ in Corollary \ref{Cor1} (4)
all come from cubic maps; see the end of Section 4 for more details.

When $(p, q, r) = (2, 7, 3)$, $f_w$ is an automorphism of the blow-up of the projective plane
at $10$ points. This automorphism coincides with the one constructed in \cite[Appendix]{BK} and \cite[Theorem 1.1]{Mc07}.

When $(p, q, r) = (2, 6, 4)$, $w$ (or its power) seems to have been geometrically realized early
by Coble and Cossec-Dolgachev (cf.~\cite[p.~39]{Do}).

When $(p,q,r)=(3,4,3)$, the group of pseudo-automorphisms of the blowup of  $\PP^2 \times \PP^2$ at certain
collections of $7$ points has been studied in \cite{D} using different techniques. It would be interesting
to compare our construction with that in \cite{D}.

The structure of the paper is the following. In Section 2 we first recall the definition of the Weyl group $W(p,q,r)$
and of Coxeter elements. We then introduce marked cubic curves, we define an action of $W(p,q,r)$
on the markings and we study some properties of this action. In Section 3 we state Theorem \ref{ThB}
which will be used in the proof of Theorem \ref{ThA}.
In Section 4 we prove Theorems \ref{ThB} and \ref{ThA}  in the case $p=2$ and we also study some aspects of the geometry
of $X_{2,q,4}$ and of the pseudo-automorphism $f_w$. In the last Section 5 we complete the proof of Theorems \ref{ThA} and \ref{ThB}
for all $p\geq 2$.

We remark that in  Oguiso-Perroni \cite{OP}, the authors have constructed
automorphisms (of product type) of positive entropy and even with Siegel disks on the product
of McMullen's rational surface and a toric variety.
Moreover, after the arXiv version of the present paper appeared,
Oguiso and Troung \cite{OT} constructed very interesting examples of primitive automorphisms of positive entropy
on certain rational or Calabi-Yau $3$-folds.

\par \vskip 1pc \noindent
{\bf Acknowledgement.}
The present work took place when the second author was visiting Bayreuth in October 2011
and in the realm of
the DFG Forschergruppe 790 Classification of algebraic surfaces and
compact complex manifolds and was partly supported by an ARF of NUS.
We express our thanks to Professor
Catanese for his interest, warm encouragement and hospitality,
and Professor Dolgachev for bringing the very interesting paper \cite{D} to our attention.
The second author would like to thank Max Planck Institute for Mathematics, Bonn, for the warm hospitality,
Professor T. -C. Dinh for clarifying the relation between
entropy and dynamical degrees and
Professor McMullen for the discussion on the Salem numbers.

\section{Preliminaries}\label{Pri}

\begin{setup}\label{Weyl}
{\rm Weyl groups and roots (cf. \cite{Hu})}
\end{setup}

Let $p \ge 2$, $q \ge 2$ and $r \ge 3$ be integers.
Let
$$n := p + q + r - 2 .$$
We now define the {\it root system} $L_n$ of type
$T_{p, q, r}$.
Let
$$\Lambda = \Lambda_n = \Z h_1+ \cdots + \Z h_{p-1} + \Z e_1 + \cdots + \Z e_{q+r}$$
be the lattice of rank $n+1$ with basis
$$h_1, h_2, \dots, h_{p-1}, e_1, \dots , e_{q+r} .$$
Late on in Section \ref{p=2}, we treat the case $p = 2$ and set $e_0 = h_1$.
The following equations define an {\it inner product}
on $\Lambda$ (cf. \cite[\S 3]{Mu2}, \cite[\S 2]{D}):
$$\begin{aligned}
h_i^2 &= h_i \cdot h_i = r-2 \,\, (1 \le i < p), \, \\
h_i \cdot h_j &= r-1 \,\, (i \ne j), \,\, h_i \cdot e_j = 0 , \\
e_i^2 &= e_i \cdot e_i = -1 \,\, (1 \le i \le q+r), \,
e_i \cdot e_j =0  \,\, (i \not= j).
\end{aligned}$$
Set
$$\kappa := r \, \sum_{i=1}^{p-1} h_i - ((p-1)(r-1) - 1) \, \sum_{j=1}^{q + r} e_j.$$
We will see that $\kappa$ corresponds to the anti-canonical divisor of some blowup $X$ of
$(\PP^{r-1})^{p-1}$ at $q+r$ points, and $\Lambda_n$ is isomorphic to $H^2(X, \Z)$.
The {\it root system} (of type $T_{p, q, r}$) is
$$
L_n := \kappa^{\perp} \cap \Lambda_n = \{\alpha \in \Lambda_n \, | \, \alpha \cdot \kappa = 0\} \, .
$$
The {\it simple roots}:
$$\begin{aligned}
\beta_1 &= -h_1 + h_2, \, \beta_2 = -h_2 + h_3, \, \dots, \beta_{p-2} = -h_{p-2} + h_{p-1}, \, \\
\alpha_0 &= h_1 - \sum_{i=1}^r e_i, \,\, \alpha_1 = e_1 - e_2, \,\, \alpha_2 = e_2 - e_3, \,\, \dots, \,\,
\alpha_{q+r-1} = e_{q+r-1} - e_{q+r}
\end{aligned}$$
form a basis of $L_n$.
The corresponding Dynkin diagram is shown in Figure 1.

\begin{center}
\setlength{\unitlength}{0.8cm}
\begin{picture}(8,3)
\thicklines

\put(-0.3,0.99){\makebox{$\dots$}}

\put(-1.2,1){\circle{0.2}}
\put(-1.3,1.5){\makebox{$\alpha_1$}}

\put(-1.1,1){\line(1,0){0.4}}

\put(-0.5,0.99){\makebox{$\dots$}}

\put(0.5,1){\line(1,0){0.4}}

\put(1,1){\circle{0.2}}

\put(1.1,1){\line(1,0){0.8}}

\put(2,1){\circle{0.2}}
\put(1.85,1.5){\makebox{$\alpha_r$}}

\put(2,0.1){\line(0,1){0.76}}
\put(2,0){\circle{0.2}}
\put(2.2,-0.1){\makebox{$\alpha_0$}}

\put(2,-0.1){\line(0,-1){0.8}}
\put(2,-1){\circle{0.2}}
\put(2.2,-1){\makebox{$\beta_1$}}
\put(2,-1.1){\line(0,-1){0.4}}
\put(1.94,-2){\makebox{$\vdots$}}

\put(2,-2.1){\line(0,-1){0.4}}
\put(2,-2.6){\circle{0.2}}
\put(2.2,-2.5){\makebox{$\beta_{p-2}$}}

\put(2.1,1){\line(1,0){0.8}}
\put(3,1){\circle{0.2}}
\put(3.1,1){\line(1,0){0.8}}
\put(4,1){\circle{0.2}}
\put(4.1,1){\line(1,0){0.4}}
\put(4.7,0.99){\makebox{$\dots$}}
\put(5.5,1){\line(1,0){0.4}}
\put(6,1){\circle{0.2}}
\put(5.7,1.5){\makebox{$\alpha_{q+r-1}$}}
\put(1,-3.5){\makebox{Figure 1. }}
\end{picture}
\end{center}

\vspace{3cm}

Any $\alpha \in L_n$ with $\alpha^2 = -2$ determines the {\it reflection}
$r_{\alpha}\in O(L_n)$ given by:
$$
r_{\alpha}(x) = x + (x \cdot \alpha) \alpha.
$$
For distinct $i, j \ge 1$,
$r_{e_i - e_j}$ (resp. $r_{h_i - h_j}$)
is the {\it transposition} interchanging the basis elements $e_i$ and $e_j$
(resp. $h_i$ and $h_j$) while fixing the other $e_k$'s and $h_{\ell}$'s.
For any $1 \le k < p$ and subset $I \subseteq \{1, 2, \dots, n\}$ with $|\, I\, | = r$,
we define the `root'
$$\alpha_{I, k} = h_k - \sum_{i \in I} e_i \in L_n$$
and the reflection (called {\it a Cremona involution}):
$$
r_{I, k} := r_{\alpha_{I, k}} \, .
$$
Its action on $\Lambda$ is given as follows:
$$\begin{aligned}
r_{I, k}(h_k) &= h_k + (h_k \cdot \alpha_{I, k}) \alpha_{I, k} = (r-1) h_k - (r-2) \sum_{i \in I} e_i, \\
r_{I, k}(h_i) &= h_i + (h_i \cdot \alpha_{I, k}) \alpha_{I, k} = (r-1)h_k + h_i - (r-1) \sum_{j \in I} e_j \,\,\, (i \ne k), \\
r_{I, k}(e_i) &= e_i + \alpha_{I,k} \,\,\, (i \in I), \\
r_{I, k}(e_j) &= e_j \,\,\, (j \not\in I).
\end{aligned}
$$

The {\it Weyl group}
$$W := W(p, q, r) = W(T_{p,q,r}) \subset O(L_n) \subset O(\Lambda_n)$$
is the subgroup of $O(L_n)$ generated by the reflections
$$
r_{\beta_i} \, (1 \le i \le p-2), \,\, r_{\alpha_j} \,\, (0 \le j < q+r) \, .
$$
Elements in the set below are called (real) {\it roots}
$$\Delta_n := \{w(\beta_i), \, w(\alpha_j) \, | \, w \in W, \, 1 \le i \le p-2, \, 0 \le j < q+r\} .$$

\begin{definition}\label{cox}
A Coxeter element $w$ of $W$ is the product $w = \prod_{i=1}^n r_{\gamma_i}$
where
$$\{\gamma_i\}_{i=1}^n = \{\beta_i\}_{i=1}^{p-2} \cup \{\alpha_j\}_{j=0}^{q+r-1}$$
as sets.
When $p = 2$,
choose $(\gamma_1, \dots, \gamma_n) = (\alpha_1, \dots, \alpha_{q+r-1}, \alpha_0)$,
we get
$$w = (12 \dots n) r_{I, 1}$$
with $I = \{1, 2, \dots, r\}$,
the product of a permutation (on $e_1, \dots, e_n$) and a Cremona involution.
This Coxeter element will be also denoted by $w_{2,q,r}$.
\end{definition}

\begin{remark}
Coxeter elements are conjugate to each other, since the Dynkin diagram $T_{p, q, r}$ is a tree
(cf. \cite[\S3.16,  \S 8.14]{Hu}).
\end{remark}

\begin{setup}\label{marked curves}
{\rm Marked cuspidal curves}
\end{setup}

Let $$C = \{YZ^2 = X^3\} \subset \PP^2$$
be the plane {\it cuspidal curve} (of arithmetic genus $1$). Consider the subset
$$
\Lambda_C \, \subset \, \left( \Pic^r(C)\right)^{p-1} \times C^{q+r}, \,\, \text{or equivalently} \,\,
\Lambda_C \, \subset \, \left( \Pic^r(C)\right)^{p-1} \times (\Pic^1(C))^{q+r} \, , \, r\geq 3
$$
consisting of $(n+1)$-tuples $$(D; c) := (D_1, \dots , D_{p-1}; c_1, \dots, c_{q+r})$$
with $c_i$ contained in the smooth locus $C \setminus \{(0, 1, 0)\}$ of $C$.

Given $(D; c)  \in \Lambda_C$,
define a {\it marking} on $C$
$$\rho = \rho_{(D; c)} : \Lambda \to \Pic(C)$$
by setting
$$
\rho(h_i) = D_i, \,\,\rho(e_j) = [c_j].
$$
Here a {\it marking} is a group homomorphism $\rho \colon \Lambda \to \Pic (C)$ such that
$\rho (h_i)\in \Pic^r (C)$ and $\rho (e_j)=[p_j]$, with $p_j\in C \setminus \{(0, 1, 0)\}$.

\begin{remark}\label{pair-mark}
The $(n+1)$-tuple  $(D; c)  \in \Lambda_C$ and the marking
$\rho = \rho_{(D; c)}$ on $C$ determine each other uniquely.
\end{remark}

As observed in \cite[Proposition 4.1, Theorem 4.3]{Mc07},
since $\Aut(C)$ acts transitively on
$$\Pic^0(C) \cong \C$$
and for any $u \in \Lambda$
$$\deg(\rho(u)) =  \frac{1}{((r-1)(p-1)-1)}(\kappa \cdot u),$$
we have:

\begin{lemma}\label{rho0}
$\rho$ is determined, up to isomorphism, by its restriction
$$\rho_0 : \Ker(\deg \circ \rho) = L_n \to \Pic^0(C) .$$
\end{lemma}

\par \noindent
Here two markings $\rho$ and $\rho'$ are {\it isomorphic} if there is an
$f\in \Aut (C)$ such that $f^*\circ \rho=\rho'$.

Set
$$U_C := \{(D; c)\in \Lambda_C \, | \, \rho_{(D; c)}(\alpha) \not= 0, \, \forall \, \alpha \in \, \Delta_n\}.$$

As observed in \cite[Example 3]{Mu}, applying the defining condition of $U_C$ to the roots $\alpha = e_i - e_j$,
and $\alpha_{I,k}$ with $|I| = r$, we have:

\begin{remark}\label{non-deg}
If $(D; c) \in U_C$, then
$c_i \ne c_j$ ($i \ne j$), and $\sum_{i \in I} c_i \not\in |D_k|$ for any $I$ with $|\, I \, | =r$
and any $k \in \{1,\dots , p-1\}$,
i.e., no $r$ points of $P(k)_{i} := \Phi_{|D_k|}(c_i) \in \PP^{r-1}$, for $k$  fixed, are contained in a hyperplane of $\PP^{r-1}$.
Here $ \Phi_{|D_k|}\colon C \to \PP^{r-1}$ is the embedding determined by $D_k$ (cf. Lemma \ref{RR} below).
\end{remark}

\begin{definition}\label{wpair}
Using markings, there is an action of $W$  on $\Lambda_C$. It is defined  by the formula (cf. Remark \ref{pair-mark}):
$$
\rho_{w(D;c)}:=\rho_{(D;c)}\circ w \, .
$$
\end{definition}

Thus $W$ acts on $U_C$ because $w(\Delta_n) = \Delta_n$. Namely, we have:

\begin{lemma}
If $w \in W$ and $(D; c) \in U_C$ then $w (D;c) \in U_C$.
\end{lemma}

\begin{setup}\label{resp}
{\rm The correspondence between vectors of $\Lambda_n \otimes \C$ and markings on $C$}
\end{setup}

Let
$$v =\sum_{i=1}^{p-1} \xi_i h_i + \sum_{j=1}^{q+r} \eta_j e_j \in \Lambda_n \otimes \C .$$
We will define an $(n+1)$-tuple $(D^v; c^v)$ in the following way.
Let
$$p(t) = (t, t^3, 1) \in C$$
be a parametrization.
Define $t_j$, $c^v_j$ and $D^v_i$ ($1 \le i < p$), by
\begin{eqnarray}\label{1}
\begin{aligned}
r \, t_0 &= v \cdot h_1 = (r-2) \xi_1 + (r-1) \sum_{i =2}^{p-1} \xi_{i} , \\
t_j   &= v \cdot e_j = -\eta_j \,\, (1 \le j \le q+r) ,  \\
c^v_j &= p(t_j - t_0) \in C, \\
D^v_i &= [rp(0) + \, p(\xi_1) -p( \xi_i)] \in \Pic^{r}(C) .
\end{aligned}
\end{eqnarray}
In this way  we get the $(n+1)$-{\it tuple}:
$$(D^v; c^v) := (D^v_1, \dots , D^v_{p-1}; \, c^v_1, \dots , c^v_{q+r}) \, \in \,
(\Pic^r(C))^{p-1} \, \times \, C^{q+r} .$$

Then $(D^v; c^v)$ determines a marking $\rho^v$ on $C$ by setting
$\rho^v(h_i) = D^v_i$, $\rho^v(e_j) = [c_j^v]$.

\begin{lemma}\label{rhoint} (see also \cite[Theorem 7.5]{Mc07})
The restriction $\rho_0^v : L_n \to \Pic_0(C) \cong \C$ of $\rho^v$
satisfies:
$$\rho_0^v(u) = (u \cdot v) [p(1) - p(0)] .$$
Hence for a root $\alpha \in \Delta_n$,
we have $\rho^v(\alpha) = 0$ if and only if $\alpha \cdot v = 0$.
In particular, the $(n+1)$-tuple $(D^v; c^v) \in U_C$ if and only if
$0 \not\in \Delta_n \cdot v$.
\end{lemma}
\begin{proof}
Direct computations show that the formula above is true  for the elements
$\beta_i$ ($1 \le i \le p-2$), $\alpha_j$ ($0 \le j < q+r$) as defined in \ref{Weyl}. This proves the result
since these elements form a basis of $L_n$.
\end{proof}

\begin{remark}\label{D-v}
Conversely, for any $(n+1)$-tuple $(D; c)$, we can use the equations in \ref{resp}
to define a vector $v$
such that $(D; c) = (D^v, c^v)$.
\end{remark}

\begin{lemma}\label{w-v}
For any $w \in W$, we have $\rho^v \circ w^{-1} = \rho^{w(v)}$,
and $w^{-1}(D^v, c^v) = (D^{w(v)}; c^{w(v)})$.
\end{lemma}

\begin{proof}
The first part follows from Lemma \ref{rhoint} and Remark \ref{rho0},
since $w \in O(\Lambda_n)$.
The second follows from the first and Definition \ref{wpair}
(cf. Remark \ref{pair-mark}).
\end{proof}

\begin{lemma}\label{equiv} $($cf. \cite[Corollary 7.7]{Mc07}$)$
Let $u,v \in \Lambda_n \otimes \C$ with
$\Delta_n \cdot u \not\ni 0 \not\in \Delta_n \cdot v$.
Then
$$
u = a v + b \kappa \, \Longleftrightarrow \, (D^u; c^u) \cong (D^v; c^v)\, .
$$
\end{lemma}

\begin{proof}
The $(n+1)$-tuples are determined by their markings on $C$ or equivalently by their restrictions on $L_n$
(cf. Remarks \ref{pair-mark} and \ref{rho0}), while the latter is determined by the inner product
on $L_n = \Lambda \cap \kappa^{\perp}$ (cf. Lemma \ref{rhoint}). The lemma follows
since $\Aut(C)$ acts on $\Pic^0(C)$ by scalar multiplication.
\end{proof}

\begin{setup}\label{lead}
{\rm
Let $w\in W$ with {\it spectral radius} $\rho(w)>1$. When $\frac{1}{p}+\frac{1}{q}+\frac{1}{r}<1$,
the root system $L_n$ is hyperbolic ($\kappa^2<0$). Hence
$\rho(w)$ is a Salem number
and $\det(xI-w)=S(x)\cdot C(x)$, where $S(x)$ is a Salem polynomial (cf.~\cite[Proposition 7.1]{Mc02}).
We say that $\lambda\in \C$ is a {\it leading eigenvalue} if $S(\lambda)=0$.
So $\rho(w)$ is a leading eigenvalue.
We say that $v \in L_n \otimes \C$ is a {\it leading eigenvector} if $w(v)=\lambda v$ with $\lambda$
a leading eigenvalue.
}
\end{setup}

\begin{proposition}\label{McTh}
Let $r \ge 3$.
Let $v \in L_n \otimes \C = (\Lambda \otimes \C) \cap \kappa^{\perp}$
be an eigenvector of some $w \in W$ with eigenvalue $\lambda$.
Then $0 \not\in \Delta_n \cdot v$, i.e. $(D^v, c^v) \in U_C$ in the sense of
Lemma $\ref{rhoint}$,
if either one of the following
two conditions is satisfied.
\begin{itemize}
\item [(1)]
$w$ is a Coxeter element and $v$ is a leading eigenvector.
\item[(2)]
$\lambda$ is not a root of unity; and $w$ has no periodic roots, i.e., no positive power of $w$
fixes a root in $\Delta_n$.
\end{itemize}
\end{proposition}

\begin{proof}
The results follow from the calculation in \cite[Theorems 2.6 and 2.7]{Mc07},
as our diagram is bipartite; see also \cite[Discussions before Theorem 1.3 and after Theorem 3.1]{Mc02}.
Indeed, in (1),
the root system $L_n$ is hyperbolic of signature $(1, n-1)$.
\end{proof}

\begin{remark}\label{pqr}
(1) happens exactly when $\frac{1}{p}+\frac{1}{q}+\frac{1}{r}<1$
(cf. \cite[Table 5]{Mc02}).
\end{remark}

\section{Main Theorem}

The following result will be used to prove Theorem \ref{ThA}.
The proof is contained in the next sections.

\begin{theorem}\label{ThB}
Let $w$ be an element of the Weyl group $W = W(p, q, r)$, with $r \ge 3$.
Let $v \in L_n \otimes \C = (\Lambda \otimes \C) \cap \kappa^{\perp}$
be an eigenvector of $w$ with $w(v) = \lambda v$.
Assume that  $0 \not\in \Delta_n \cdot v$, i.e., $(D^v, c^v) \in U_C$
in the sense of Lemma $\ref{rhoint}$.
Then there exist a blowup
$$X = X_{p, q, r} \,\, \to \,\, (\PP^{r-1})^{p-1} =
\PP^{r-1} \times \cdots \times \PP^{r-1}$$
at $q+r$ points $P_i$ lying on the  cuspidal curve $\Phi_{|D^v|}(C) \subset (\PP^{r-1})^{p-1}$
of multi-degree $r$ and a pseudo-automorphism
$f_w : X \,\, \dashrightarrow \,\, X$
such that $(f_w)^* | H^2(X, \Z)$ equals $w$.

Further, if $|\lambda| > 1$, then $\lambda$ is equal to $|\lambda|$, the spectral radius $\rho(w)$ of $w$ and also
the first dynamical degree $d_1(f_w)$ of $f_w$.
\end{theorem}

\section{Proof of Theorems when $p = 2$}\label{p=2}

We will frequently use the following result.

\begin{lemma}\label{RR}
Let $C$ be the cuspidal curve of arithmetic genus $1$
and let $D$ be a Cartier divisor on $C$ of degree $r$.
\begin{itemize}
\item [(1)]
If $r = 1$, then there is a unique smooth point $P$ of $C$
such that $P \sim D$ $($linear equivalence$)$.
\item[(2)]
If $r = \deg(D) \ge 3$, then the complete linear system
$|D|$ is base point free and defines an
embedding $\Phi_{|D|} : C \to \PP^{r-1}$.
\end{itemize}
\end{lemma}

\begin{proof}
By the Riemann-Roch theorem (true for all projective curves as in Hartshorne's book, Ch IV, Ex 1.9) and Serre duality
for Cohen-Macaulay projective variety,
we have $h^0(C, \mathcal{O}_C(D)) = r$. The result follows.
Indeed, the second part of (1) is worked out in Hartshorne's book, Ch II, Example 6.11.4.
\end{proof}

We now prove Theorem \ref{ThB} when $p = 2$.
In the definition of the lattice $\Lambda_n$ and $L_n$, we set $p = 2$ and $e_0 = h_1$.
Let $(D;c)\in U_C$ and
consider the embedding
$$\Phi_{|D|} : C \to \PP^{r-1}$$
given by the base-point free
complete linear system $|D|$. Set $P_i := \Phi_{|D|}(c_i)$.
Let
$$\pi_{(D; c)} : X = X_{(D; c)} \to \PP^{r-1}$$ be the blowup of the $n$ points $P_i$
with $E_i = \pi_{(D; c)}^{-1}(P_i)$. For any $w \in W$, set $(D';c'):=w(D;c)$ and
define similarly $\Phi_{|D'|}$, $P_i'$, $\pi_{(D'; c')} : X' = X_{(D'; c')} \to \PP^{r-1}$, $E_i'$.

The result below should be well known
but we work it out since we need to extend it
to the case $p > 2$ in Section \ref{p>2};
see \cite[VI, Proposition 1, page 86]{DO}.
Our statement also incorporates the marking on the curve $C$ embedded in $\PP^{r-1}$.

\begin{proposition}\label{Cal}
Let $p = 2$.
Let $w\in W$ and $(D;c)\in U_C$. Define $(D';c'):=w(D;c)$. Consider the blowups
$$
\pi\colon X_{(D;c)}\to  \PP^{r-1}\, , \quad \pi' \colon X_{(D';c')}\to \PP^{r-1}
$$
at the points $P_i=\Phi_{|D|}(c_i)$ $($resp. $P_i'=\Phi_{|D'|}(c_i'))$.

Then there exists a pseudo-isomorphism $f_w\colon X_{(D;c)} \dashrightarrow X_{(D';c')}$ such that
$$
f_w^*\colon H^2(X_{(D';c')},\Z)\to H^2(X_{(D;c)},\Z)
$$
coincides with $w$ after the identifications $[E_j']=e_j=[E_j]$, $j\geq 1$, $\pi'^*[H]=e_0=\pi^*[H]$.
Here $H$ is the hyperplane of $\PP^{r-1}$
and $E_i$ $($resp. $E_i')$ is the exceptional divisor over $P_i$ $($resp. $P_i')$.
\end{proposition}

\begin{proof}
Since $W$ is generated by the transpositions $r_{e_i - e_j}$ and the Cremona involution
$r_{I, 1}$, we need to prove the result only when $w$ is one of them.

Our proof is top down: first construct a pseudo-isomorphism $X = X_{(D; c)} \dashrightarrow X'$ and then show that
$X'$ equals $X_{(D'; c')}$, the blowup of $\PP^{r-1}$ at the $n$ points $\Phi_{|D'|}(c_i')$.

Consider first the case $w = r_{I, 1}$ with $I = \{1,2, \dots, r\}$.
Let
$$X_P = X_{P_1, \dots, P_r} \to \PP^{r-1}$$
be the blowup of the $r$ points $P_i$ ($1 \le i \le r$).
Since these $r$ points $P_i$ span the whole space (cf. Remark \ref{non-deg}), we can take
the standard Cremona involution
$$\Psi_P = \Psi_{P_1, \dots, P_r} :{\PP}^{r-1}  \dashrightarrow {\PP}^{r-1} .$$
$\Psi_P$ is given by the linear system
$$|\OO_{\PP^{r-1}}(r-1) - (r-2)\sum_{i=1}^r P_i| .$$
A basis of this linear system is:
$$\sum_{j\not=i}H_j, \,\,\,\, i\in \{ 1, \dots , r\}$$
where
$H_i$ is the hyperplane passing through $r-1$ points $\{P_1, \dots, P_{i-1}, P_{i+1}, \dots, P_r\}$.
The base locus of the linear system (the place where $\Psi_P$ is not defined)
is the union of $H_i \cap H_j$ ($1 \le i < j \le r$).
Using new coordinate system so that $P_1 = [1: 0: \dots : 0],
\dots, P_r = [0: \dots : 0 : 1]$, our $\Psi_P$ is given by
$$\Psi_P : [X_1: \dots : X_r] \,\, \to \,\, [\frac{1}{X_1}: \dots : \frac{1}{X_r}].$$
Let $E_i \subset X_P$ be the inverse image of $P_i$ and $E_0 \subset X_P$ the total transform of a hyperplane of $\PP^{r-1}$.
Then it is known that $\Psi_P$ lifts to an involutive pseudo-automorphism
$$\widetilde{\Psi_P} : X_P \to X_P$$
exchanging $E_i$ with the proper transform $H_i' \subset X_P$ of $H_i$ (cf. \cite{DO}).
This means that
$$\widetilde{\Psi_P}^*E_i = H_i' \sim E_0 - \sum_{j \ne i} E_j
\,\,\,\, \text{(linear equivalence)}$$
Denote by $e_i = [E_i] \in H^2(X_P, \Z)$.
Then
$$\widetilde{\Psi_P}^*e_i = [\widetilde{\Psi_P}^*E_i] = [E_0 - \sum_{j \ne i} E_j] =
e_0 - \sum_{j \ne i} e_j = e_i + (e_0 - \sum_{i=1}^r e_i) = w(e_i).$$
By the definition of the Cremona involution in terms of the linear system,
$$\widetilde{\Psi_P}^* E_0 = (r-1)E_0 - (r-2)\sum_{i=1}^r E_i$$
and hence
$\widetilde{\Psi_P}^*e_0 = w(e_0)$.

The blowup $X_P \to \PP^{r-1}$ is centered at $r$ smooth points $P_i = \Phi_{|D|}(c_i)$,
and hence gives an isomorphism between the proper transform $C_{X} \subset X_P$ of $C$ and $C$.
Since $C = \Phi_{|D|}(C) \subset \PP^{r-1}$ is a non-degenerate curve,
it is not contained in any hyperplane $H_i$.
Hence $C_X$ is not contained in $H_i'$.
Now
$$\deg(H_i' | C_{X}) = \deg(E_0 - \sum_{j \ne i} E_j) | C_{X}) = \deg(\OO_{\PP^{r-1}}(1) | C) - (r-1) = 1$$
since $C = \Phi_{|D|}(C)$ is a curve of degree $r$ in $\PP^{r-1}$.
Thus $C_X$ meets $H_i'$ only at one point and transversally.
Since the Cremona involution $\Psi_P : \PP^{r-1} \dasharrow \PP^{r-1}$
blows up $r$ smooth points $P_i$ on $C$ and collapses $H_i'$ to a point called $P_i'$ in the codomain $\PP^{r-1}$,
it maps $C \subset \PP^{r-1}$ isomorphically to a curve $C'$ in the codomain $\PP^{r-1}$.
As sets, we have $\{P_i'\} = \{P_i\}$. This $C'$ is also the isomorphic image of $C_X \subset X_P$
via the map
$$X_P \overset{\widetilde{\Psi_P}}{\dashrightarrow} X_P \to \PP^{r-1} .$$
This isomorphism of curves factors as
$$C_X \overset{\widetilde{\Psi_P}}{\to} C'_X \to C' .$$

Let us calculate the very ample divisor $D' = \OO_{\PP^{r-1}}(1) | C'$ giving rise to the
embedding $C' \subset \PP^{r-1}$. By the above identification $C_X = C'_X = C'$ and further the identification
$C_X = \Phi_{|D|}(C) = C$, we have
$$D' = E_0 | C'_X = \widetilde{\Psi_P}^* E_0 | C_X =
((r-1)E_0 - (r-2)\sum_{i=1}^r E_i) | C_X = (r-1)D - (r-2) \sum_{i=1}^r c_i = w(D)$$
(cf. Definition \ref{wpair}).
Let $c_i' \in C'$ be the preimage of the point $P_i' \in \PP^{r-1}$ via the embedding $\Phi_{|D'|} : C' \to \PP^{r-1}$.
Under the same identification $C = \Phi_{|D|}(C) = C_X = C'_X = C'$,
we have (cf. Lemma \ref{RR}):
$$C = C' \ni c_i' = P_i' = H_i' \, | \, C_X \sim (E_0 - \sum_{j \ne i} E_j) | C_X =
D - \sum_{j \ne i}c_j = w(c_i).$$

For $r+1 \le j \le n$, the point
$P_j$ is not contained in the indeterminacy set: the union of $H_i \cap H_j$,
otherwise, the $r$ points $P_1, \dots, P_{i-1}, P_j, P_{i+1}, \dots, P_r$ are contained in the hyperplane
$H_i$, contradicting Remark \ref{non-deg}.
Let $Q_j$ ($r+1 \le j \le n$) be the $\Psi_P$-image of $P_j$.
For $1 \le i \le r$, set $Q_i = P_i$.
Let
$$\pi_{(D; c)} : X = X_{(D; c)} \to \PP^{r-1}$$
be the blowup of the $n$ points $P_i$,
$E_0 \subset X$ the pullback of the hyperplane of $\PP^{r-1}$,
$E_i = \pi_{(D; c)}^{-1}(P_i)$ ($i \ge 1$), and $e_i$ ($i \ge 0$) the cohomology class of $E_i$ in $H^2(X, \Z)$.
Let
$$\pi': X' \to \PP^{r-1}$$
be the blowup of the $n$ points $Q_i$,
$E_0' \subset X'$ the pullback of the hyperplane of $\PP^{r-1}$,
$E_i' = (\pi')^{-1}(Q_i)$ ($i \ge 1$), and $e_i'$ ($i \ge 0$) the cohomology class of $E_i'$ in $H^2(X', \Z)$.
Then $\widetilde{\Psi_P}$ lifts to a pseudo-isomorphism
$$f_w : X \to X' .$$
Identify $H^2(X_P, \Z)$ with its embedded image (via pullback) in $H^2(X, \Z)$.
By the calculation above and the construction, we have $f_w^* e_i' = w(e_i)$ for all $i \le r$
and $f_w^*(E_j') = E_j$ ($j > r$) (so $f_w^*(e_j') = e_j = w(e_j)$), if we identify $H^2(X', \Z) = H^2(X, \Z)$
by letting $e_i = e_i'$ ($i \ge 0$); thus $f_w^* = w$.

By the argument above, if we set $(D'; c') = w(D; c)$, then the above $\pi' : X' \to \PP^{r-1}$ is just the blowup
of $n$ points $P_i' = \Phi_{|D'|}(c_i')$ on the curve $C' = \Phi_{|D'|}(C') \subset \PP^{r-1}$,
i.e., it is $\pi_{(D'; c')}$.
This proves Proposition \ref{Cal} when $w$ is a Cremona involution.

\par \vskip 1pc
Next, consider the case where $w = r_{e_a - e_b}$ is a transposition of the basis elements $e_a$ and $e_b$
and fixing the others.
Take an automorphism $\sigma$ of $\PP^{r-1}$ interchanging two points $P_a$ and $P_b$.
Let $C' = \sigma(C) \subset \sigma(\PP^{r-1}) = \PP^{r-1}$.
Set $P_a' = P_a$, $P_b' = P_b$ and $P_j' = \sigma(P_j)$ ($j \ne a, b$).
Let
$$X' \to \PP^{r-1}$$
be the blowup of the $n$ points $P_i'$
with $E_i'$ the inverse of $P_i'$.
Then $\sigma$ lifts to an isomorphism
$$f_w: X \to X' .$$
We see that $f_w^* = w$ if we identify $H^2(X', \Z) = H^2(X, \Z)$ by letting $[E_i'] = e_i = [E_i]$ as above.
Define $(D'; c')$ so that $D' = D$, $c_a' = c_a, c_b' = c_b$ and $c_j' = \sigma(c_j)$ ($j \ne a, b$).
Using the identification $C = C_X = C_X' = C'$ as above, we obtain
$(D'; c') = w(D; c)$. This implies Proposition \ref{Cal} as in the previous case.
\end{proof}

\begin{setup}
{\bf Proof of Theorem \ref{ThB} when $p = 2$}
\end{setup}

Given $v$ as in Theorem \ref{ThB}, we define $(D^v; c^v)$ as in \ref{resp} (cf.~Lemma \ref{rhoint}).
Set $(D; c) = (D^v; c^v)$.
Then we get the pseudo-isomorphism
$$f_w : X = X_{(D; c)} \to X' = X_{(D'; c')}$$
as in Proposition \ref{Cal}
with $(D'; c') = w(D; c)$ and $f_w^* = w$ on $H^2(X', \Z)$ identified with $H^2(X, \Z)$
by letting $[E_i'] = [E_i]$ and $[(\pi')^*H'] = [\pi^*H]$.
By Lemmas \ref{w-v} and \ref{equiv},
$$(D'; c') = w(D^v; c^v) = (D^{w^{-1}(v)}; c^{w^{-1}(v)}) = (D^{\lambda^{-1}v}; c^{\lambda^{-1}v})
= (D^v; c^v) = (D; c)$$
(up to the action of $\Aut(C)$).
Thus we
may identify the $\pi_{(D'; c')} : X_{(D'; c')} \to \PP^{r-1}$ in Proposition \ref{Cal}
with $\pi_{(D; c)} : X = X_{(D; c)} \to \PP^{r-1}$ so that $f_w$ is a pseudo-automorphism.
This proves Theorem \ref{ThB}.
Indeed, for the final part (when $|\lambda| > 1$),
the Coxeter system is hyperbolic, so $\lambda$ is the largest root of a Salem polynomial
and also the spectral radius $\rho(w)$ of $w$ (cf. \cite[Proposition 7.1]{Mc02}). Thus
$$d_1(f_w) = \rho(f_w^* | H^2(X, \Z)) = \rho(w) = \lambda$$
by the definition of $d_1(f_w)$ (cf. \cite{DS}).

\begin{setup}
{\bf Proof of Theorem \ref{ThA} when $p = 2$}
\end{setup}
Theorem \ref{ThA}  follows from Proposition \ref{McTh} (1), Theorem \ref{ThB} and its proof,
by taking $\lambda$ in Proposition \ref{McTh} to be the spectral radius of $w$;
see also \ref{lead} and Remark \ref{pqr}.

\begin{setup}\label{constr}
{\rm Concrete construction of $f_w$ on $X_{2, q, r}$ as in Theorem \ref{ThA}}
\end{setup}

We first construct a pseudo-automorphism $f$ such that $f_* = w$ where $w = (12 \cdots n) r_{I, 1}$
is a Coxeter element of the root system $L_n$ of type $T_{2, n-r, r}$ (cf. Definition \ref{cox}).
Then $f_w = f^{-1}$ meets the requirement. To do so,
take an eigenvector $v$ of $w$ such that $w(v) = \lambda v$ and $\lambda$
is the spectral radius of $w \in O(L_n)$ (which turns out
to be $d_1(f_w)$, since $f_{w}^* = w$).

Define the $(n+1)$-tuple $(D; c) = (D^v; c^v)$ as in \ref{resp}.
Let $P_i = \Phi_{|D|}(c_i) \in \Phi_{|D|}(C) \subset \PP^{r-1}$.
Choose a new coordinate system of $\PP^{r-1}$ such that $P_1 = [1 : 0 :  \dots : 0], \dots, P_r = [0 : \dots : 0 : 1]$.
Consider the standard Cremona involution:
$$\gamma : \PP^{r-1} \to \PP^{r-1}, \,\, [X_1 : \dots : X_r] \mapsto [\frac{1}{X_1} : \dots : \frac{1}{X_r}] .$$

Let
$$\pi = \pi_{(D; c)} : X = X_{(D; c)} \to \PP^{r-1}$$
be the blowup at the $n$ points $P_i$
and let $E_i = \pi^{-1}(P_i)$ and $E_0 \subset X$ the total transform of a hyperplane of $\PP^{r-1}$.
Then by the proof of Proposition \ref{Cal} and Theorem \ref{ThB},
there is a projective automorphism $g$ of $\PP^{r-1}$
such that
$g \circ \gamma$ lifts to a pseudo-isomorphism
$$f = f_{(12 \cdots n)} \circ f_{r_{I,1}}: X \to X'$$
where
$f_{(12 \cdots n)}$ is the lifting of $g$ and so is an isomorphism. Moreover,
$$f_* = (f_{(12 \cdots n)})_* (f_{r_{I,1}})_* = w$$
on $H^2(X, \Z)$ identified with $H^2(X', \Z)$ by letting $[E_i] = e_i = [E_i']$.
Recall that $X'=X_{(D';c')}$ is the blowup at $n$ points $P_i' = \Phi_{|D'|}(c_i') \in \PP^{r-1}$ and since $v$ is an eigenvector
we have
$$X' = X_{(D'; c')} = X_{w(D^v; c^v)} = X_{(D^{w^{-1}(v)}; c^{w^{-1}(v)})} = X_{(D^{\lambda^{-1}v}; c^{\lambda^{-1}v})}
= X_{(D^v; c^v)} = X$$
(up to isomorphism).
Now the identification $(D'; c') = w(D; c)$ with $(D; c)$ and the fact that
$w(c_i) = c_{i+1}$ (mod $n$) for $i > r$ force $(g \circ \gamma)(P_i) = P_{i+1}$ (mod $n$).
Conversely, if we can find $g$ as above
then we can forget about the eigenvector $v$ or so, and straightaway say that $(g \circ \gamma)^{-1}$ lifts to
a pseudo-automorphism $f_w$ on the blowup $X \to \PP^{r-1}$ at the $n$ points $P_i$
which satisfies the conclusion of Theorem \ref{ThA}.

\begin{remark}\label{lift}
When $p = 2$, our $f_w$ in Theorem \ref{ThA} lifts to an isomorphism.
Indeed, by the construction in \ref{constr},
it is enough to lift $f = f_{(12 \cdots n)} \circ f_{r_{I,1}} : X \dashrightarrow X$ to an isomorphism.
By \cite[VI, Lemma 1]{DO}, there is a  further blowup $\sigma: X_1 \to X$ and a blowup $X_2\to \PP^{r-1}$
such that $f_{r_{I,1}}$ lifts to an isomorphism
$f_1 : X_1 \to X_2$. We can take a corresponding blowup $X_3 \to X$ of the images of the centers
lying below the exceptional divisors of $X_2\to \PP^{r-1}$
to lift the isomorphism $f_{(12 \cdots n)}$ to an isomorphism $f_2 : X_2 \to X_3$.
Now the isomorphism $f_3 = f_2 \circ f_1 : X_1 \to X_3$
(resp. $f_4 = f_3^{-1}$) is a lifting of $f$ (resp. $f^{-1} = f_w$).
\end{remark}

\medskip

\begin{setup}
{\bf On the geometry of $X_{2,q,4}$ with $q \ge 5$}
\end{setup}

In this subsection, we prove the following:

\begin{proposition}\label{X254}
Let $w=w_{2,q,4}$ $(q \ge 5)$. Let $f_w$ be the pseudo-automorphism of $X:=X_{2,q,4}$ in Theorem $\ref{ThA}$
and $C_X\subset X$ the proper transform of $C_D:=\Phi_{|D|}(C)\subset \PP^3$.
Then:
\begin{itemize}
\item[(1)]
$f_w$  stabilizes the cuspidal curve $C_X$ and permutes
members $F_t'$ of the rational pencil $|{-}K_X/2|$ each of which is a strict transform
of an {\it irreducible} quadric hypersurface $F_t \subset \PP^3$ with $F_t' \cap F_{t'}' = C_X$ $(t \ne t')$.
Moreover, all the quadrics $F_t$, except two: $F_i$ $(i = 1, 2)$, are smooth.
\item[(2)]
$f_w$ stabilizes the blowup $F_1'$ of the quadric cone $F_1$ whose vertex $P$ is the cusp of $C_D$.
When $q = 5$, every effective divisor $E$ with the class $[E]$ fixed by $f_w^*$ is a union of members in $|{-}K_X/2|$.
\item[(3)]
$S := F_1'$ $($the $(q+4)$-point blowup of the quadric $F_1)$
is disjoint from the indeterminacy of $f_w$.
The restriction
$f_S := f_w \, | \, S$ is a well defined automorphism of $S$
with $d_1(f_S) = d_1(f_w) > 1$.
\end{itemize}
\end{proposition}

\begin{proof}
By the proof, $f_w(C_X) = C_X$ holds in Theorem \ref{ThA} for any $(p, q, r)$.
Since $C_D$ has arithmetic genus $1$ and degree $4$, it is contained in a linear system
$|\mathcal{I}(2)|$ of quadrics of dimension $\geq 1$.
This follows from the long cohomology sequence associated to
$$
0\to \mathcal{I}(2) \to \mathcal{O}_{\PP^3}(2) \to \mathcal{O}_{C_D}(2)\to 0\, .
$$
Alternatively, we may assume that $C_D = \{ (1,z,z^2,z^4) \, : \, z \in \C \}$ in new coordinates.
By a direct calculation, $|\mathcal{I}(2)|$ is a pencil spanned by
$$F_1 := \{ X_2^2 = X_1X_3\}, \,\,\,\, F_2 := \{X_3^2 = X_1X_4\} ,$$
every member $F_t\not= F_i$ ($i = 1, 2$) is smooth,
and $(\Sing(F_i)) \cap (C_D \setminus \Sing(C_D)) = \emptyset$.

Let
$$\pi : X \to \PP^3$$
be the blowup at the $q+4$ points $P_i$ as in Theorem \ref{ThA}
with $E_i = \pi^{-1}(P_i)$ and $E_0 \subset X$ the total transform of a hyperplane of $\PP^3$.
For $F\in |\mathcal{I}(2)|$, the proper transform $F'$ of $F$
satisfies
$$F' \sim 2E_0 - \sum_{i=1}^{q+4} E_i \,\,\, \text{(linear equivalence)} ,$$
so
$$-K_X = -(\pi^*K_{\PP^3} + 2\sum_{i=1}^{q+4} E_i) \sim 2F' .$$
Since $-K_X$ is preserved by $f_w$, we have $2(f_w^*F' - F') \sim 0$, so
$f_w^*F' - F' \sim 0$, because the rational manifold $X$ is simply connected and hence
cohomologous divisors are just linear equivalent divisors.

$F'$, or equivalent $F = \pi(F')$, is irreducible. Otherwise,
$F = L_1 \cup L_2$ with two hyperplanes $L_i$. Since all $(q+4)$ points $P_i \in C_D$ belong to $F$,
we may assume that $L_1$ contains $5$ of $P_i$. This contradicts Remark \ref{non-deg}
(cf. Proposition \ref{McTh}).
For two distinct such $F$, say $F_t$, $F_{t'}$,
the intersection $F_t \cap F_{t'}$ includes $C_D$ and hence equals $C_D$ by comparing the degree.
This proves (1).

If $E$ is a divisor whose class $[E]$ is fixed by $f_w^*$ (e.g., $E = a F'$), then
either $\dim |E| \le 0$ or $|E|$ is composed of a pencil, otherwise, $f_w$
would descend to a surface or threefold automorphism of the first dynamical degree equal to $1$
via a fibration with general fiber a curve or a point, contradicting the fact that $d_1(f_w) > 1$
(cf. \cite{DN}).
In particular, $|aF'|$ ($a > 0$) is composed of a pencil (necessarily parametrized by a curve $B \cong \PP^1$
because the irregularity $q(B) \le q(X) = 0$) stabilized by $f_w$.
The induced action of $f_w$ on $B \cong \PP^1$ has at least one fixed point. Namely,
at least one $F'_0 \in |F'|$ is $f_w$-stable.

When $q = 5$, the characteristic polynomial of $f_w^* | H^2(X, \Z)$ has the form
$$\phi_8(x) (x+1)(x-1)$$
(cf. \cite[Table 5]{Mc02}), where $x-1$ corresponds to the $f_w$-invariant class $\kappa = [-K_X] = 2[F']$.
If $E$ is an integral divisor with $f_w^*[E] = [E]$ then $b E \sim a F'$ for some coprime integers $a, b$.
Since
$$[F'] \cdot [F'] = (\kappa)^2/4 = 4-q = -1$$
we get $b = \pm 1$. In particular, every effective
divisor $E$
with $[E]$ fixed by $f_w^*$ is a member of the pencil $|aF'|$ and hence equal to a union of
$F_t' \in |F'|$ by the Stein factorization.

As in \ref{lift} or \cite{DO}, $f_w: X \dashrightarrow X$ (and $f_{r_{I, 1}}$)
is well-defined outside the proper transforms $H_{ij}' := H_i' \cap H_j'$ of the lines $H_{ij} := H_i \cap H_j$
($1 \le i < j \le 4$) (cf. the notation of Proposition \ref{Cal}).
Our $F_1' \in |F'|$ has the characteristic property as being the only singular
member in $|F'|$ whose singular point $\pi^{-1}(P)$
is the cusp of $C_X$ ($P$ being the vertex of $F_1$).
Since
$$H_i' . H_j' . F_t' = (E_0 - \sum_{i \ne \ell=1}^4 E_{\ell}) .
(E_0 - \sum_{j \ne \ell=1}^4 E_{\ell}) . (2E_0 - \sum_{\ell=1}^{q+4} E_{\ell})
= 2 (E_0^3) - \sum_{i, j \ne \ell=1}^4 (E_{\ell})^3 = 2 - 2 = 0$$
$H_{ij}'$ is either contained in $F_t'$ or disjoint from $F_t'$.
If $H_{ij}'$ is contained in $F_1'$, then the line
$H_{ij}$ is contained in the cone $F_1$ and passes through its vertex $P$, and
$H_{i}$ intersects
the non-degenerate curve $C_D$ at its cusp $P$ and three points $P_j$ ($j \ne i, 1 \le j \le 4$),
hence
$$4 = \deg(C_D) = C_D . H_i \ge 2 + 3 ,$$ a contradiction.
Thus no $H_{ij}'$ intersects $F_1'$.
Our $f_w$ is well defined at $S := F_1'$, and
$f_w(F_1') \in |F'|$ satisfies the same characteristic property as $F_1'$
and hence equals $F_1'$.
The isomorphism $f_4 : X_3 \to X_1$ in Remark \ref{lift} (lifting $f_w$)
is just $f_w$ around $S$ and hence $f_S = f_w | S$ is an automorphism of $S$.
This proves (2).

Using the lifting $f_4$ of $f_w$, we have
$$f_S^* (L_{f_w} | S) = d_1(f_w) L_{f_w} | S$$
where $L_{f_w}$ is the eigenvector of
$f_w^* | H^2(X, \Z) = w$ corresponding to the eigenvalue $d_1(f_w) = \rho(w)$.
To prove (3), we only need to show $L_{f_w} | F_t' \ne 0$.
To do so, write
$$L_{f_w} = v = \sum_{i=0}^{q+4} v_i e_i$$
as in \ref{resp}.
Then
$$
(L_{f_w} | F_t') . (E_i | F_t') = L_{f_w} . E_i . F_t' = (\sum_{j=0}^{q+4} v_j E_j) . E_i . (2E_0 - \sum_{j=1}^{q+4} E_j)
= - v_i (E_i)^3 = - v_i \, .
$$
Since $e_0$ is not an eigenvector of $w_{2,q,4}$, it follows that $L_{f_w} | F_t' \ne 0$.
\end{proof}

\begin{remark}
The Salem number $ 1.28064 \dots$ is also realized in
\cite{Mc07} as $d_1(f_{13})$ of an automorphism $f_{13}$ on the blowup $X_{13}$ of $\PP^2$ at $13$ points on a cubic curve.
The map $f_S$ in Proposition \ref{X254} (3) with $q = 5$ is not the descent of $f_{13}$ because the characteristic polynomial
of $f_{13}^* | H^2(X_{13}, \Z)$ is of the form $\phi_8(x)(x^4+1)(x^2-1)$.
Since $W(2,5,4)$ can be embedded in $W(2, 7, 3)$,
our Proposition \ref{X254} (3) with $q = 5$ is compatible with \cite{Mc07}.
\end{remark}

As a consequence of Theorem \ref{ThA} and Proposition \ref{X254}, we have:
\begin{corollary}\label{h254}
Let $f_w\colon X_{2,q,4}  \dashrightarrow X_{2,q,4}$ be
as in Theorem $\ref{ThA}$ with $q \ge 5$ and let $S = F_1'$ be as in Proposition $\ref{X254}$.
Then the topological entropy $h(f_w)$ of $f_w$ satisfies:
$$
h(f_w) = h(f_S) = \log d_1(f_S) = \log d_1(f_w) > 0.
$$
\end{corollary}

\begin{proof}
By the Poincar\'e duality and noting that $d_1(f_w)$ is a Salem number,
we have $d_1(f_w) = d_2(f_w)$.
Thus
$$\log d_1(f_w) \ge h(f_w) \ge h(f_S) = \log d_1(f_S) = \log d_1(f_w)$$
(cf. \cite{DS}, \cite{G}, \cite{Yo}, and taking an equivariant resolution of $S$), and we are done.
\end{proof}

\section{Proof of Theorems for all $p \ge 2$}\label{p>2}

We now prove Theorem \ref{ThB} for $p \ge 3$.
Let $w \in W$.
Let $(D; c) \in U_C$. Denote by $(D'; c') = w(D; c)$.
Consider the embedding:
$$\Phi_{(D; c)} : C \to (\PP^{r-1})^{p-1}, \,\,\,
(x \mapsto (\Phi_{|D_1|}(x), \dots, \Phi_{|D_{p-1}|}(x))) .$$
Set $P_j = (\Phi_{|D_1|}(c_j), \dots, \Phi_{|D_{p-1}|}(c_j))$.
Let
$$\pi_{(D; c)} : X = X_{(D; c)} \to (\PP^{r-1})^{p-1}$$ be the blowup at the $q+r$ points $P_j$
with $E_j = \pi_{(D; c)}^{-1}(P_j)$.
Similarly, we define $\Phi_{(D'; c')}$, $P_j'$, $\pi_{(D'; c')} : X' = X_{(D'; c')} \to (\PP^{r-1})^{p-1}$, $E_j'$.

For the result below, see \cite[Theorem 1]{Mu2}, \cite{D}. Our statement also incorporates the marking on the curve $C$
embedded in $(\PP^{r-1})^{p-1}$.

\begin{proposition}\label{Cal2}
Suppose that $w \in W$ and $(D; c) \in U_C$. Then there is a pseudo-isomorphism
$f_w : X \to X' = X_{(D', c')}$
such that $f_w^* = w$ if we identify
$$H^2(X, \Z) = \sum_{i=1}^{p-1} \Z h_i + \sum_{j=1}^{q+r} \Z e_j  = H^2(X', \Z)$$
by letting $[E_j] = e_j = [E_j']$ $(j \ge 1)$ and
$$[\pi_{(D; c)}^*\OO_{\PP_i^{r-1}}(1)] = h_i = [\pi_{(D'; c')}^*\OO_{\PP_i^{r-1}}(1)]$$
where $\PP_i^{r-1}$ is the $i$-th factor of the product $(\PP^{r-1})^{p-1}$.
\end{proposition}

\begin{proof}
The proof is similar to Proposition \ref{Cal}.
Since the Weyl group is generated by the reflections
$r_{h_i - h_j}$ (resp. $r_{e_i - e_j}$) corresponding to the exchange of the factors $\PP_i^{r-1}$ and $\PP_j^{r-1}$
(resp. $P_i$ and $P_j$ of the blowup),
and the Cremona involution $r_{\alpha_0}$ with $\alpha_0 = h_1 - \sum_{i=1}^r e_i$, we have only to
consider the case $w = r_{\alpha_0}$.
This $w$ is realized by the lifting
$$f_w : X \to X'$$
of the
following standard (geometric) Cremona involution (cf. \cite[Lemma in \S 3]{Mu2}):
$$\begin{aligned}
\Psi \, : \, & \, (\PP^{r-1})^{p-1} \, \to \, (\PP^{r-1})^{p-1}, \\
& \, ([X_1 : \dots : X_r], [Y_1 : \dots : Y_r], \dots, [Z_1 : \dots : Z_r]) \, \mapsto \, \\
& \, ([\frac{1}{X_1} : \dots : \frac{1}{X_r}],
[\frac{Y_1}{X_1} : \dots : \frac{Y_r}{X_r}], \dots, [\frac{Z_1}{X_1} : \dots : \frac{Z_r}{X_r}]) .
\end{aligned}$$
Here, with new coordinates, we may assume that $P_1, \dots, P_r$ are images of the
standard vertices $[1 : 0 : \dots : 0], \dots, [0 : \dots : 0 : 1]$ in $\PP^{r-1}$
via the diagonal embedding
$$\PP^{r-1} \, \to \, (\PP^{r-1})^{p-1}, \,\,\,\, P \, \mapsto \, (P, \dots, P)$$
and
$$X \to (\PP^{r-1})^{p-1}$$
is the blowup of $q+r$ points $P_i$
and
$$X' \to (\PP^{r-1})^{p-1}$$
is the blowup of $Q_1 := P_1, \dots, Q_r := P_r$ and $Q_j := \Psi(P_j)$ ($r < j \le q+r$).
By the form of the map,
$$f_w^* h_1 = (r-1) h_1 - (r-2) \sum_{i=1}^{r} [E_i] = w(h_1)$$
if we identify $H^2(X, \Z) = H^2(X', \Z)$ (here and below) by letting $h_i = h_i', [E_j] = e_j = [E_j']$.
Here and below  $E_i \subset X$ (resp. $E_i' \subset X'$) is the inverse of $P_i$ (resp. $Q_i$),
$h_i$ (resp. $h_i'$) is the (cohomology class of) total transform of the hyperplane
$\OO_{\PP_i^{r-1}}(1)$ of the $i$-th factor of the domain
(resp. codomain) of $\Psi$.
From the form of $\Psi$, we have also
$$f_w^* h_i = (r-1) h_1 - (r-2+1) \sum_{i=1}^{r} [E_i] + h_i = w(h_i) \,\, (1 \le i < p)$$
where the $h_i$'s in the middle of the display and the extra $1$ in $r-2+1$ correspond to the numerators $Y_1, \dots, Z_r$ in
the defining rational functions of $\Psi$.
As observed in \cite[Lemma in \S 3]{Mu2},
using the affine coordinates
$$((x_2, \dots, x_r), (y_2, \dots, y_r), \dots, (z_2, \dots, z_r))$$
of $(\PP^{r-1})^{p-1}$
around the point $P_1$ (the diagonal image of the point $[1 : 0 : \dots : 0] \in \PP^{r-1}$),
the map
$$X \overset{f_w}\to X' \to (\PP^{r-1})^{p-1}$$
takes the following form around $E_1$:
$$\begin{aligned}
E_1 \, \ni \, & \, ((x_2, \dots, x_r), (y_2, \dots, y_r), \dots, (z_2, \dots, z_r) \\
\, \mapsto \,
& \, ([0 : \frac{1}{x_2} : \dots : \frac{1}{x_r}],
[1 : \frac{y_2}{x_2} : \dots : \frac{y_r}{x_r}], \dots, [1 : \frac{z_2}{x_2} : \dots : \frac{z_r}{x_r}]) .
\end{aligned}$$
Hence for the hyperplane $H_{1i} \subset (\PP^{r-1})^{p-1}$ defined by $X_i = 0$, its
proper transform $H_{1i}' \subset X'$ satisfies
(when $i = 1$)
$f_w^*H_{1i}' = E_i$.
Since $\Psi$ is an involution and by a similar observation, for all $1 \le i \le r$, we have
(noting that $[H_{1i}] = h_1$):
$$[f_w^* E_i'] = [H_{1i}'] = h_1 - \sum_{i \ne j=1}^r [E_j] = w(e_i)$$
if we identify $H^2(X, \Z) = H^2(X', \Z)$ as above.
The equality $f_w^* e_j' = e_j$ ($r < j \le q+r$) is by the definition of $Q_j$.
Thus we have $f_w^* = w$ on $H^2(X', \Z)$ (identified with $H^2(X, \Z)$).

To check that $X' \to (\PP^{r-1})^{p-1}$ is just the blowup of points $P_i'$
determined by the $(n+1)$-tuple $w(D; c)$,
we can argue as in Proposition \ref{Cal}.
Indeed, let $C_X \subset X$ be the proper transform of $C = \Phi_{(D; c)}(C) \subset (\PP^{r-1})^{p-1}$
(which is isomorphic to $C$ since we blow up only smooth points on $C$).
Then for $1 \le i \le r$,
we have
$$\deg(H_{1i}' | C_X) = \deg(H_{1i} | \Phi_{(D; c)}(C)) - \deg(\sum_{i \ne j=1}^r E_j) | C_X = r - (r-1) = 1 .$$
Hence $C_X$ meets $H_{1i}'$ at only one point and transversally.
So the map $X \overset{f_w}\to X' \to (\PP^{r-1})^{p-1}$ collapses $H_{1i}'$ to a smooth point $Q_i$ on the image $C'$
of $C$ which is
contained in the codomain $(\PP^{r-1})^{p-1}$  of the Cremona involution $\Psi$.
With the identification $C' = C_X = \Phi_{(D; c)}(C) = C$, we have
$$[\OO_{\PP_i^{r-1}}(1) | C'] = ((r-1) h_1 - (r-1) \sum_{i=1}^{r} [E_i] + h_i) | C_X = w(h_i) | C = w(D)_i = D'_i$$
which is a degree $r \ge 3$ (very ample) divisor and embeds $C'$ onto $C_i'$ ($\subset \PP_i^{r-1}$,
the $i$-th factor of the codomain of $\Psi$).
With the identification $C' = C_X = \Phi_{(D; c)}(C) = C = C_1
:= \Phi_{|D_1|}(C)$ ($\subset \PP_1^{r-1}$,
the first factor of the domain of $\Psi$),
the point $Q_i \in C'$ is given by
$$[H_{1i}' | C_X] = [H_{1i} | C_1] - \sum_{i \ne j=1}^r E_j | C_X
= D_1 - \sum_{i \ne j=1}^r c_j = w(c_i) \in C .$$
Hence $Q_i$ is one of $P_i'$ ($1 \le i \le r$)
defined before Proposition \ref{Cal2}. By the construction, $Q_j = \Psi(P_j) = P_j'$ for $r < j \le q+r$.
Thus $X' = X_{(D'; c')}$.
This proves Proposition \ref{Cal2}.
\end{proof}

\begin{setup}
{\bf Proof of Theorems \ref{ThB} and \ref{ThA}}
\end{setup}

The same argument for $p = 2$ now works for all $p \ge 2$, but with Proposition \ref{Cal}
replaced by Proposition \ref{Cal2}.

\begin{setup}
{\bf On the geometry of $X_{3,q,3}$ with $q \ge 4$}
\end{setup}

\begin{proposition}
Let $X_{3,q,3}$ $(q \ge 4)$ and $f_w$ be as in Theorem $\ref{ThA}$.
Then $X_{3,q,3}$ is the blowup of $\PP^2 \times \PP^2$ at $q+3$ points, and $f_w$ permutes
members of the linear system $|{-}K_X/3|$ of dimension $\ge 2$.
When $q = 4$, every divisor $E$ with class $[E]$
fixed by $f_w^*$ satisfies $E \sim a(-K_X/3)$ $(${\it linear equivalence}$)$ for some $a \in \Z$.
\end{proposition}

The proof is similar to that of Proposition \ref{X254} and is left to the reader.

%
%
%

\end{document}